\documentclass[review]{elsarticle}
\usepackage{amsmath}
\usepackage{array}
\usepackage{amsthm}
\usepackage{amssymb}
\usepackage{amsfonts}
\usepackage{mathtools}
\usepackage{bm}
\usepackage{float,latexsym,moreverb,color,amscd,graphicx,amsmath,epstopdf,epsfig,mathptmx}
\usepackage{amsthm,amssymb,mathrsfs,mathtools,amsmath,dsfont}
\usepackage[colorlinks,bookmarksopen,bookmarksnumbered,citecolor=red,urlcolor=red]{hyperref}
\usepackage{lineno}
\usepackage{subfigure}

\newtheorem{example}{Example}
\newtheorem{theo}{Theorem}
\newtheorem{lemma}{Lemma}
\newtheorem{remark}{Remark}

\journal{........................}

\makeatletter
\newcommand{\thickhline}{%
\noalign {\ifnum 0=`}\fi \hrule height 3pt
\futurelet \reserved@a \@xhline
}
\makeatother

\usepackage{numcompress}
\begin{document}

 \begin{frontmatter}

 \title{An improved bound on Legendre approximation}

\author[firstaddress]{Mehdi Hamzehnejad}
\author[secendaddress]{Mohammad Mehdi Hosseini}
\author[secendaddress]{Abbas Salemi\corref{mycorrespondingauthor}}
\cortext[mycorrespondingauthor]{Corresponding author}
\ead{salemi@uk.ac.ir}
\address[firstaddress]{Department of Mathematic, Kerman Graduate University of Advanced Technology, Kerman, Iran.}
\address[secendaddress]{ Department of Applied Mathematics and Mahani Mathematical Research Center, Shahid Bahonar University of Kerman, Kerman, Iran.}
\begin{abstract}
In this paper, new relations between the derivatives of the Legendre polynomials are obtained, and by these relations, new upper bounds for the Legendre coefficients of differentiable functions are presented. 
These upper bounds are sharp and 
cover more categories of differentiable functions. Moreover, 
 new and sharper bounds for the approximation error of the partial sums of Legendre polynomials are provided. Numerical examples are given to validate our theoretical results.
 \end{abstract}

 \begin{keyword}
Legendre polynomials\sep Legendre expansions\sep approximation error \sep error bound. 
\MSC[2010] 41A25\sep 41A10.
\end{keyword}

\end{frontmatter}


 \section{Introduction}
 One of the most important properties of Legendre polynomials $ L_n(x),n=0,1,\ldots $ 
is their ability to approximate a function which has fascinated a great attention in recent years. See for example:  \cite{Su64,Shu12,Xi13,Xi20,Hai18,Ham22,Hai21}.
Let us suppose that a suitably smooth function $f$ has the following expansion:
\begin{equation}\label{exp}
f(x)=\sum_{n=0}^{\infty}a_nL_n(x),
\end{equation}
where $ a_n, $ Legendre coefficient, is  defined by
\begin{equation*}
a_n=\left(n+\frac{1}{2}\right)\int_{-1}^{1}f(x)L_n(x)dx.
\end{equation*}
Consider the approximation of differentiable functions using the partial sum
\[ f_N(x)=\sum_{n=0}^{N}a_nL_n(x).\]
Several papers are available in literature dealing with this topic under different smoothness assumptions of differentiable functions.

 Suetin proved that if the function $ f $ has $ r $ continuous derivatives on $ [-1,1] $ and $ f^{(r)}\in Lip \alpha $ with $ r+\alpha\geq\frac{1}{2} $, then \citep{Su64}
\begin{equation*}\label{St}
\Vert f(x)-f_N(x)\Vert_{\infty}\leq \mathit{C}\frac{log N}{N^{r+\alpha-\frac{1}{2}}},\quad -1\leq x\leq1,
\end{equation*}
where $ \mathit{C} $ is positive constant independent of $ n $.

Later on, in \citep{Shu12} the authors obtained error estimates which depend on the Legendre coefficients of the function $ f $. They showed that if $ f, f^{\prime},\cdots,f^{(r-1)} $ are absolutely continuous and the $ r^{th} $ derivative $ f^{(r)} $ is of bounded variation and 
\begin{equation} 
V_r=\int_{-1}^{1}\frac{\vert f^{(r+1)}(x)\vert}{\sqrt{1-x^2}}dx<\infty,
\label{VR1}
\end{equation}
then
\begin{equation}
\Vert f(x)-f_N(x)\Vert_{\infty}\leq\sqrt{\frac{\pi}{2}}\frac{V_r}{(r-1)\sqrt{N-r}\prod_{j=2}^{r}2N-2j+3}.
\label{wel}
\end{equation}

Under the assumptions stated in \citep{Shu12}, Wang \citep{Hai18} proved that if
\begin{equation}
\widehat{V}_r=\int_{-1}^{1}\frac{\vert f^{(r+1)}(x)\vert}{\sqrt[4]{1-x^2}}dx<\infty,
\label{VR}
\end{equation}
then
\begin{equation}
\Vert f(x)-f_N(x)\Vert_{\infty}\leq \left\{
		\begin{array}{lr}
	\frac{4\widehat{V}_1}{\sqrt{\pi(2N-5)}}&r=1\\
			\frac{2\widehat{V}_r}{\prod_{j=2}^{r}(N-j+\frac{1}{2})(r-1)\sqrt{\pi(2N-2r-1)}}& r\geq2.
		\end{array}\right.
	\label{wang}
\end{equation}
In all the above estimates the rate of convergence is essentially $ N^{-r+\frac{1}{2}} $ that in each case the function $ f $ depends on certain smoothness properties.

Now we want to obtain a new upper bound for the Legendre coefficients without the assumptions in (\ref{VR1}) and (\ref{VR}) and provide new approximation error for more categories of differentiable functions. For this purpose, we state new relationship  between derivatives of Legendre polynomials. Moreover, by using the telescoping series and by induction, we obtain new error estimates in uniform norm which are sharper than previous bounds presented in \cite{Shu12, Hai18, Hai21}.

The rest of this paper is organized as follows. In Section \ref{sec2}, using some properties of Legendre polynomials, we obtain a new expansion for derivatives of Legendre polynomials and by using this expansion, we derive a new upper bound for the Legendre coefficients of differentiable functions. Also, by using the asymptotic of the Legendre coefficients, we provide new approximation errors of a function $ f (x) $ by partial sums of Legendre polynomials. In Section \ref{sec3}, numerical results are presented and we compare the proposed upper bound with the upper bounds were presented so far. 
\section{New approximation errors of Legendre polynomials}\label{sec2}
In this section some necessary properties of Legendre polynomials are presented.
Then we provide a new expansion for derivative of Legendre polynomials. Using this expansion, new approximation errors of the partial sums of the Legendre polynomials are provided. 

Legendre polynomials are the eigenfunctions of the following Sturm-Liouville problem \citep{Shen11}
\begin{equation}\label{Str}
\left(\left(1-x^2\right)L_n^{\prime}(x)\right)^{\prime}=-n(n+1)L_n(x).
\end{equation}
Also, the following relation is hold \citep{Shen11}
\begin{equation}\label{rab}
\left(1-x^2\right)L_n^{\prime}(x)=\frac{n(n+1)}{2n+1}\left(L_{n-1}(x)-L_{n+1}(x)\right).
\end{equation}
The following orthogonality is hold on $ \left[ -1,1\right]$:
\begin{equation*}
\int_{-1}^{1}L_n(x)L_m(x)dx=\frac{\delta_{mn}}{n+\frac{1}{2}},
\end{equation*}
where $ \delta_{mn} $ is the Kronecker delta.

By using Sturm-Liouville problem and relation (\ref{rab}), we provide a new expansion for derivative of Legendre polynomials. For simplicity, we use the following notation.
\begin{equation}\label{sim}
\mathcal{L}_n(x):=\frac{\left(1-x^2\right)L_n^{\prime}(x)}{n(n+1)}.
\end{equation}
The following key lemma present a new relation between derivatives of Legendre polynomials.
\begin{lemma}\label{Key}
Let $\mathcal{L}_n(x) $ be as in (\ref{sim}). Then $\mathcal{L}_n(x)$ has the following representation, for $ 1\leq r \leq n-1$. 
\begin{align}\label{Ind4}
\mathcal{L}_n(x)&=\left(\frac{(-1)^r\mathcal{L}_{n-r}(x)}{\lambda^1_{n,r}}+\frac{(-1)^{r-1}\mathcal{L}_{n-r+2}(x)}{\lambda^2_{n,r}}+ \cdots \right. \nonumber\\
&\left.-\frac{\mathcal{L}_{n+r-2}(x)}{\lambda^{2^{r-1}}_{n,r}}
+\frac{\mathcal{L}_{n+r}(x)}{\lambda^{2^{r}}_{n,r}}\right)^{(r)},
\end{align}
where
\begin{equation*}\nonumber
 \lambda^j_{n, r}= \left\{
 \begin{array}{lr}
 \prod_{i=1}^{r}(2n-2i+3)&j=1, 2\\
 \prod_{i=1}^{r}(2n+2i-1)&j=2^{r}-1, 2^r
 \end{array}\right.
 \end{equation*} 
 Moreover, for $1\leq j \leq 2^{r-1}$,
 \[\lambda^1_{n, r}=\lambda^2_{n, r} \leq \lambda^{2j-1}_{n, r} = \lambda^{2j}_{n, r} \leq \lambda^{2^{r}-1}_{n, r}=\lambda^{2^{r}}_{n, r}.
\]
\end{lemma}
\begin{proof}
By combining (\ref{Str}) and (\ref{sim}), we obtain
\begin{equation}\label{LLL1}
\left(\mathcal{L}_{n}(x)\right)^{\prime}=-L_n(x).
\end{equation}
If we plug (\ref{LLL1}) in (\ref{rab}), we get
\begin{equation}\label{LLL}
\mathcal{L}_{n}(x)=\frac{1}{2n+1}\left(-\mathcal{L}_{n-1}(x)+\mathcal{L}_{n+1}(x)\right)'.
\end{equation}
Then for $r=1$, $\mathcal{L}_{n}(x)$ has the following respresentation.
\begin{equation}\label{comb}
\mathcal{L}_n(x)=\left(-\frac{\mathcal{L}_{n-1}(x)}{(2n+1)} +\frac{\mathcal{L}_{n+1}(x)}{(2n+1)}\right)^{(1)}
=\left(-\frac{\mathcal{L}_{n-1}(x)}{\lambda^1_{n,1}} +\frac{\mathcal{L}_{n+1}(x)}{\lambda^2_{n,1}}\right)^{(1)}, 
\end{equation}
where $\lambda^1_{n,1}=\lambda^2_{n,1}.$
By using (\ref{LLL}) in two numerators of (\ref{comb}), we get the following representation for $\mathcal{L}_n(x)$
\begin{align}\label{comb2}
\mathcal{L}_n(x)&=\left(\frac{\mathcal{L}_{n-2}(x)}{\lambda^1_{n,1}(2n-1)}-\frac{\mathcal{L}_{n}(x)}{\lambda^1_{n,1}(2n-1)}
-\frac{\mathcal{L}_{n}(x)}{\lambda^2_{n,1}(2n+3)}+\frac{\mathcal{L}_{n+2}(x)}{\lambda^2_{n,1}(2n+3)}\right)^{(2)}\nonumber\\
&=\left(\frac{\mathcal{L}_{n-2}(x)}{\lambda^1_{n,2}}-\frac{\mathcal{L}_{n}(x)}{\lambda^2_{n,2}}
-\frac{\mathcal{L}_{n}(x)}{\lambda^3_{n,2}}+\frac{\mathcal{L}_{n+2}(x)}{\lambda^4_{n,2}}\right)^{(2)},
\end{align}
where $\lambda^1_{n,2}=\lambda^2_{n,2}\leq \lambda^3_{n,2}=\lambda^4_{n,2} .$
By induction, we assume that (\ref{Ind4}) holds for $1\leq r \leq n-2$. We will show that (\ref{Ind4}) holds for $r+1$.
Using (\ref{LLL}) in all numerators of (\ref{Ind4}), we get the following representation for $\mathcal{L}_n(x)$ with $2^{r+1}$ terms as follows:
\begin{align}
\mathcal{L}_n(x)&=\left(\frac{(-1)^{r+1}\mathcal{L}_{n-r-1}(x)}{\lambda^1_{n, r}(2n-2r+1)}+\frac{(-1)^{r}\mathcal{L}_{n-r+1}(x)}{\lambda^1_{n, r}(2n-2r+1)}+ \cdots \right. \nonumber\\
&\left.-\frac{\mathcal{L}_{n+r-1}(x)}{\lambda^{2^{r}}_{n, r}(2n+2r+1)}
+\frac{\mathcal{L}_{n+r+1}(x)}{\lambda^{2^{r}}_{n, r}(2n+2r+1)}\right)^{(r+1)},\nonumber
\end{align}
Easy computation shows that
\begin{equation}\nonumber
\lambda^{2j-1}_{n,r+1}=\lambda^{2j}_{n,r+1}= \left\{
 \begin{array}{lr}
 \lambda^{1}_{n, r} (2n-2r+1)& j=1\\
\lambda^{j}_{n,r} (\gamma_j)&2\leq j\leq2^{r}-1\\
\lambda^{2^r}_{n, r} (2n+2r+1)& j=2^{r}
 \end{array}\right.
 \end{equation} 
where $2n-2r+3\leq \gamma_j \leq 2n+2r-1.$ 
Also, by assumption, we know that 
$\lambda^{1}_{n, r} \leq \lambda^{j}_{n, r} \leq \lambda^{2^r}_{n, r}.$ Then 
\[\lambda^{1}_{n, r+1} =\lambda^{2}_{n, r+1}\leq\lambda^{2j-1}_{n,r+1}=\lambda^{2j}_{n,r+1} \leq \lambda^{2^{r+1}-1}_{n, r+1}=\lambda^{2^{r+1}}_{n, r+1}, ~~ 1\leq j\leq2^{r}.\] 
Therefore, the result holds.
\end{proof}
\begin{lemma}\label{le} 
Suppose that $ S_{n}^{0}=\frac{2}{2n+1} $ and for $ 1\leq r\leq n-1 $ 
\begin{equation}
S_{n}^{r}=\frac{2}{\lambda^1_{n,r}(2n-2r+1)}+\frac{2}{\lambda^2_{n,r}(2n-2r+5)}+ \cdots +\frac{2}{\lambda^{2^{r}-1}_{n,r}(2n+2r-3)}+\frac{2}{\lambda^{2^{r}}_{n,r}(2n+2r+1)}.
\end{equation}
Then for $ r\geq1 $
\begin{equation}\label{baz}
S_{n}^{r}=\frac{1}{2n+1}\left(S_{n-1}^{r-1}+S_{n+1}^{r-1}\right).
\end{equation}
\end{lemma}
\begin{proof}
To prove relation (\ref{baz}), We can see that 
\begin{equation}
S_{n}^{1}=\frac{2}{\lambda^1_{n,1}(2n-1)}+\frac{2}{\lambda^2_{n,1}(2n+3)}=\frac{1}{2n+1}\left(\frac{2}{2n-1}+\frac{2}{2n+3}\right)=\frac{1}{2n+1}\left(S_{n-1}^{0}+S_{n+1}^{0}\right).
\end{equation}
This complete the proof for $ r=1 $. Also, for $ r=2 $ we obtain
\begin{small}
\begin{align}
S_{n}^{2}&=\frac{2}{\lambda^1_{n,2}(2n-3)}+\frac{2}{\lambda^2_{n,2}(2n+1)}+\frac{2}{\lambda^{3}_{n,2}(2n+3)}+\frac{2}{\lambda^{4}_{n,2}(2n+5)}\nonumber\\
&=\frac{1}{2n+1}\left(\frac{2}{(2n-1)(2n-3)}+\frac{2}{(2n-1)(2n+1)}+\frac{2}{(2n+3)(2n+1)}+\frac{2}{(2n+3)(2n+5)}\right)\nonumber\\
&=\frac{1}{2n+1}\left(\frac{2}{\lambda^1_{n-1,1}(2n-3)}+\frac{2}{\lambda^2_{n-1,1}(2n+1)}+\frac{2}{\lambda^{1}_{n+1,1}(2n+1)}+\frac{2}{\lambda^{2}_{n+1,1}(2n+5)}\right)\nonumber\\
&=\frac{1}{2n+1}\left(S_{n-1}^{1}+S_{n+1}^{1}\right).
\end{align}
\end{small}
This complete the proof for $ r=2 $. By induction, suppose (\ref{baz}) holds for nonnegative integer $ r $. We will show that (\ref{baz}) holds for $r+1$. 
So we have
\begin{small}
\begin{align}\label{S1}
&\frac{1}{2n+1}\left(S_{n-1}^{r}+S_{n+1}^{r}\right)\nonumber\\
&=\frac{1}{2n+1}\left(\frac{2}{\lambda^1_{n-1,r}(2n-2r-1)}+\frac{2}{\lambda^2_{n-1,r}(2n-2r+3)}+ \cdots
+\frac{2}{\lambda^{2^{r}-1}_{n-1,r}(2n+2r-5)}+\frac{2}{\lambda^{2^{r}}_{n-1,r}(2n+2r-1)}\right.\nonumber\\
&\left.+\frac{2}{\lambda^1_{n+1,r}(2n-2r+3)}+\frac{2}{\lambda^2_{n+1,r}(2n-2r+7)}+ \cdots
+\frac{2}{\lambda^{2^{r}-1}_{n+1,r}(2n+2r-1)}+\frac{2}{\lambda^{2^{r}}_{n+1,r}(2n+2r+3)}\right).
\end{align}
From Lemma \ref{Key}, for $ 1\leq j\leq2^{r-1} $ we have $ \lambda^{2j-1}_{n,r}=\lambda^{2j}_{n,r} $ . Also we know that
\begin{align}\nonumber
\lambda^1_{n,r+1}=\lambda^2_{n,r+1}=(2n+1)\lambda^1_{n-1,r},\quad\lambda^{2^r-1}_{n,r+1}=\lambda^{2^r}_{n,r+1}=(2n+1)\lambda^{2^r-1}_{n+1,r}.
\end{align}
So (\ref{S1}) becomes to
\begin{align}
&\frac{1}{2n+1}\left(S_{n-1}^{r}+S_{n+1}^{r}\right)\nonumber\\
&=\left(\frac{2}{\lambda^1_{n,r+1}(2n-2r-1)}+\frac{2}{\lambda^2_{n,r+1}(2n-2r+3)}
+ \cdots
+\frac{2}{\lambda^{2j-1}_{n,r+1}(2n+2r-5)}+\frac{2}{\lambda^{2j}_{n,r+1}(2n+2r-1)}\right.\nonumber\\
&\left.+\frac{2}{\lambda^{2i-1}_{n,r+1}(2n-2r+3)}+\frac{2}{\lambda^{2i}_{n,r+1}(2n-2r+7)}+ \cdots
+\frac{2}{\lambda^{2^{r}-1}_{n,r+1}(2n+2r-1)}+\frac{1}{\lambda^{2^{r}}_{n,r+1}(2n+2r+3)}\right)\nonumber\\
&=S_{n}^{r+1}.
\end{align}
\end{small}
which complete the proof.
\end{proof}
\begin{lemma}\label{l3}
Under the assumptions of Lemma \ref{le}, the following relation holds for $ r\geq0 $ 
\begin{equation}\label{jam1}
S_{n}^{r}=\frac{2^{r+1}}{\prod_{j=1}^{r+1}(2n-2r+4j-3)}
\end{equation}
\end{lemma}
\begin{proof}
We know that $ S_{n}^{0}=\frac{2}{2n+1} $ and 
\begin{equation}
S_{n}^{1}=\frac{1}{2n+1}\left(S_{n-1}^{0}+S_{n+1}^{0}\right)=\frac{4}{(2n-1)(2n+3)}.
\end{equation}
By induction, suppose (\ref{jam1}) holds integer $ r $. We will show that (\ref{jam1}) holds for $r+1$. Applying the relation (\ref{baz}) we obtain
\begin{align}
&\frac{1}{2n+1}\left(S_{n-1}^{r}+S_{n+1}^{r}\right)\nonumber\\
&=\frac{1}{2n+1}\left(\frac{2^{r+1}}{(2n-2r-1)\left[\prod_{j=1}^{r}(2n-2r+4j-1)\right]}
+\frac{2^{r+1}}{\left[\prod_{j=1}^{r}(2n-2r+4j-1)\right](2n+2r+3)}\right)\nonumber\\
&=\frac{1}{2n+1}\left(\frac{2^{r+1}(2n+2r+3+2n-2r-1)}{(2n-2r-1)\left[\prod_{j=1}^{r}(2n-2r+4j-1)\right](2n+2r+3)}\right)\nonumber\\
&=\frac{2^{r+2}}{\prod_{j=1}^{r+2}\left(2n-2(r+1)+4j-3\right)}=S_{n}^{r+1}.
\end{align}
Then the result holds. 
\end{proof}

Let $n \geq 1$, from \citep[eq. 27]{Antonov2010}, we see that
\begin{equation}\label{LL}
\vert L_{n-1}(x)- L_{n+1}(x)\vert<\frac{2\mathbf{A}(1-x^2)^{\frac{1}{4}}}{\sqrt{n+\frac{1}{3}}},\quad \mathbf{A}=0.825031.
\end{equation}
Also, by using (\ref{rab}) and (\ref{sim}) we obtain the following inequality of derivative of Legendre polynomials for $ x\in[-1,1]. $
\begin{equation}\label{eqb1}
\vert \mathcal{L}_n(x)\vert\leq\frac{1}{2n+1}\left(\vert L_{n-1}(x)- L_{n+1}(x)\vert\right)
\end{equation}
Therefore, by using (\ref{LL}) and (\ref{eqb1}), we get 
\begin{equation}\label{eqb}
\vert \mathcal{L}_n(x)\vert<\frac{\mathbf{A}(1-x^2)^{\frac{1}{4}}}{(n+\frac{1}{2})\sqrt{n+\frac{1}{3}}},\quad -1\leq x\leq1.
\end{equation}
In the next theorem, new upper bounds for Legendre coefficients of differentiable functions are derived.
\begin{theo}\label{Th}
Let $ f, f^{\prime},\cdots, f^{(r-1)} $ be absolutely continuous and the $ r^{th} $ derivative $ f^{(r)} $ is of bounded variation on $ [-1,1] $. Then the upper bound for the Legendre coefficients of the function $ f $ for $ n\geq r $ is as follows:
\begin{equation}\label{HCo} 
\left\vert a_{n}\right\vert<\frac{\mathbf{A}(n+\frac{1}{2})U_r}{\sqrt{n-r+\frac{1}{3}}\prod_{j=1}^{r+1}(n-r+2j-\frac{3}{2})},
\end{equation} 
where, $ U_{r}=\int_{-1}^{1}\vert f^{(r+1)}(x)\vert dx $.
\end{theo}
\begin{proof}
By using (\ref{Str}), the Legendre coefficients for the function $ f(x) $ are as follows
\begin{equation*}
a_n=\left(n+\frac{1}{2}\right)\int_{-1}^{1}f(x)L_n(x)dx=-(n+\frac{1}{2})\int_{-1}^{1}f(x)\left(\mathcal{L}_n(x)\right)^{\prime}dx.
\end{equation*}
Using Integration by parts
\begin{equation}\label{Part2}
a_{n}=-(n+\frac{1}{2})\left[f(x)\mathcal{L}_n(x)\right]_{-1}^1+(n+\frac{1}{2})\int_{-1}^{1}f^{\prime}(x)\mathcal{L}_n(x)dx.
\end{equation}
Since $ \mathcal{L}_n(1)=\mathcal{L}_n(-1)=0, $ we obtain that the first part in (\ref{Part2}) vanishes and we get
\begin{equation}\label{tas}
a_{n}=(n+\frac{1}{2})\int_{-1}^{1}f^{\prime}(x)\mathcal{L}_n(x)dx.
\end{equation}
Using (\ref{eqb}) and applying $ U_0=\int_{-1}^{1}\left\vert f^{\prime}(x)\right\vert dx $, we get
\begin{equation}\label{Bon}
\vert a_{n}\vert\leq\frac{\mathbf{A}U_0}{\sqrt{n+\frac{1}{3}}}.
\end{equation}
which complete the proof for $ r=0 $.

For $ r\geq 1 $, applying (\ref{Ind4}) in (\ref{tas}) we have
\begin{align}
a_n=(n+\frac{1}{2})\int_{-1}^1f^{\prime}(x)&\left(\frac{(-1)^r\mathcal{L}_{n-r}(x)}{\lambda^1_{n,r}}+\frac{(-1)^{r-1}\mathcal{L}_{n-r+2}(x)}{\lambda^2_{n,r}}+ \cdots \right. \nonumber\\
&\left.-\frac{\mathcal{L}_{n+r-2}(x)}{\lambda^{2^{r-1}}_{n,r}}
+\frac{\mathcal{L}_{n+r}(x)}{\lambda^{2^{r}}_{n,r}}\right)^{(r)}dx.
\end{align}
Since $ \mathcal{L}_{n+j}(1)=\mathcal{L}_{n+j}(-1)=0, -r\leq j\leq r, $ with $ r $ integrations by parts and vanishing the first term at each steps, we get 
\begin{align}
a_n=(n+\frac{1}{2})\int_{-1}^1f^{(r+1)}(x)&\left(\frac{(-1)^r\mathcal{L}_{n-r}(x)}{\lambda^1_{n,r}}+\frac{(-1)^{r-1}\mathcal{L}_{n-r+2}(x)}{\lambda^2_{n,r}}+ \cdots \right. \nonumber\\
&\left.-\frac{\mathcal{L}_{n+r-2}(x)}{\lambda^{2^{r-1}}_{n,r}}
+\frac{\mathcal{L}_{n+r}(x)}{\lambda^{2^{r}}_{n,r}}\right)dx.
\end{align}
Using (\ref{eqb}) and applying the given condition $ U_r=\int_{-1}^{1}|f^{(r+1)}(x)|dx $, we obtain the following relation with $ 2^r $ terms.
\begin{align*}
\vert a_n\vert&<(n+\frac{1}{2})\left(\frac{2\mathbf{A}}{\lambda^1_{n,r}(2n-2r+1)\sqrt{n-r+\frac{1}{3}}}+\frac{2\mathbf{A}}{\lambda^2_{n,r}(2n-2r+5)\sqrt{n-r+\frac{7}{3}}}+ \cdots \right. \nonumber\\
&\left.+\frac{2\mathbf{A}}{\lambda^{2^r-1}_{n,r}(2n+2r-3)\sqrt{n+r-\frac{5}{3}}}
+\frac{2\mathbf{A}}{\lambda^{2^r}_{n,r}(2n+2r+1)\sqrt{n+r+\frac{1}{3}}}\right)U_r\\
&\leq\frac{\mathbf{A}(n+\frac{1}{2})U_r}{\sqrt{n-r+\frac{1}{3}}}\left(\frac{2}{\lambda^1_{n,r}(2n-2r+1)}+ \cdots 
+\frac{2}{\lambda^{2^r}_{n,r}(2n+2r+1)}\right)\nonumber\\
&=\frac{\mathbf{A}(n+\frac{1}{2})U_rS_n^r}{\sqrt{n-r+\frac{1}{3}}}
\end{align*}
By Lemma \ref{l3}, we obtain the following upper bound
\begin{align*}
\vert a_n\vert&<\frac{2^{r+1}\mathbf{A}(n+\frac{1}{2})U_r}{\prod_{j=1}^{r+1}(2n-2r+4j-3)\sqrt{n-r+\frac{1}{3}}}\\
&=\frac{\mathbf{A}(n+\frac{1}{2})U_r}{\prod_{j=1}^{r+1}(n-r+2j-\frac{3}{2})\sqrt{n-r+\frac{1}{3}}}.
\end{align*}
This complete the proof.
\end{proof}

In the next theorem, new approximation errors using the partial sums of the Legendre polynomials is provided.
\begin{theo}\label{Th2}
Let $ f, f^{\prime},\cdots, f^{(r-1)} $ be absolutely continuous and the $ r^{th} $ derivative $ f^{(r)} $ is of bounded variation on $ [-1,1] $. Then for $ r\geq1 $ and $ N\geq r+1 $
\begin{equation}\label{Co}
 \Vert f(x)-f_N(x)\Vert_{\infty}\leq \left\{
 \begin{array}{lr}
 \frac{2\mathbf{A}U_1}{\sqrt{N-\frac{2}{3}}}&r=1\\
 \frac{\mathbf{A}(N+\frac{3}{2})U_r}{\prod_{j=1}^{r}(N-r+2j-\frac{3}{2})(r-\frac{1}{2})\sqrt{N-r+\frac{1}{3}}}& r~odd\\
 \frac{\mathbf{A}(N+\frac{1}{2})U_r}{\prod_{j=1}^{r}(N-r+2j-\frac{3}{2})(r-\frac{1}{2})\sqrt{N-r+\frac{1}{3}}}& r~even.
 \end{array}\right.
 \end{equation} 
\end{theo}
\begin{proof}
Applying the inequality $ \vert L_n(x)\vert\leq1 $ for all $ x\in\left[-1,1\right] $, we have
\begin{equation}\label{Sum}
\Vert f(x)-f_N(x)\Vert_{\infty}=\vert\sum_{n=N+1}^{\infty}a_nL_n(x)\vert\leq\sum_{n=N+1}^{\infty}\vert a_n\vert \vert L_n(x)\vert\leq\sum_{n=N+1}^{\infty}\vert a_n\vert.
\end{equation}
By using (\ref{HCo}) for $ r=1 $, we obtain
\begin{align}
&\Vert f(x)-f_N(x)\Vert_{\infty}\leq\sum_{n=N+1}^{\infty}\frac{\mathbf{A}(n+\frac{1}{2})U_1}{(n-\frac{1}{2})(n+\frac{3}{2})\sqrt{n-\frac{2}{3}}}\nonumber\\
&\leq\sum_{n=N+1}^{\infty}\frac{\mathbf{A}U_1}{(n-\frac{1}{2})\sqrt{n-\frac{2}{3}}}\leq\sum_{n=N+1}^{\infty}\frac{\mathbf{A}U_1}{(n-\frac{1}{2})^{\frac{3}{2}}\sqrt{1-\frac{1}{6(n-\frac{1}{2})}}}\nonumber\\
&\leq\frac{\mathbf{A}U_1}{\sqrt{1-\frac{1}{6(N-\frac{1}{2})}}}\sum_{n=N+1}^{\infty}\frac{1}{(n-\frac{1}{2})^{\frac{3}{2}}}\leq\frac{\mathbf{A}U_1}{\sqrt{1-\frac{1}{6(N-\frac{1}{2})}}}\int_{n=N}^{\infty}\frac{1}{(x-\frac{1}{2})^{\frac{3}{2}}}dx\nonumber\\
&=\frac{2\mathbf{A}U_1}{\sqrt{N-\frac{2}{3}}}
\end{align}
By using (\ref{HCo}) for odd $ r\geq3 $, we get
\begin{align}
&\Vert f(x)-f_N(x)\Vert_{\infty}\leq\sum_{n=N+1}^{\infty}\frac{\mathbf{A}(n+\frac{1}{2})U_r}{(n-r+\frac{1}{2})\cdots(n-\frac{1}{2})(n+\frac{3}{2})\cdots(n+r+\frac{1}{2})\sqrt{n-r+\frac{1}{3}}}\nonumber\\
&\leq\sum_{n=N+1}^{\infty}\frac{\mathbf{A}U_r}{(n-r+\frac{1}{2})\cdots(n-\frac{1}{2})(n+\frac{7}{2})\cdots(n+r+\frac{1}{2})\sqrt{n-r+\frac{1}{3}}}\nonumber\\
&=\mathbf{A}U_r\sum_{n=N+1}^{\infty}\frac{1}{(n+r+\frac{1}{2})^{r+\frac{1}{2}}(1-\frac{2r}{(n+r+\frac{1}{2})})\cdots(1-\frac{r+1}{(n+r+\frac{1}{2})})(1-\frac{r-3}{(n+r+\frac{1}{2})})\cdots(1-\frac{2}{(n+r+\frac{1}{2})})\sqrt{1-\frac{2r+\frac{1}{6}}{(n+r+\frac{1}{2})}}}\nonumber\\
&\leq\frac{\mathbf{A}U_r}{(1-\frac{2r}{(N+r+\frac{1}{2})})\cdots(1-\frac{r+1}{(N+r+\frac{1}{2})})(1-\frac{r-3}{(N+r+\frac{1}{2})})\cdots(1-\frac{2}{(N+r+\frac{1}{2})})\sqrt{1-\frac{2r+\frac{1}{6}}{(N+r+\frac{1}{2})}}}\sum_{n=N+1}^{\infty}\frac{1}{(n+r+\frac{1}{2})^{r+\frac{1}{2}}}\nonumber\\
&\leq\frac{\mathbf{A}U_r}{(1-\frac{2r}{(N+r+\frac{1}{2})})\cdots(1-\frac{r+1}{(N+r+\frac{1}{2})})(1-\frac{r-3}{(N+r+\frac{1}{2})})\cdots(1-\frac{2}{(N+r+\frac{1}{2})})\sqrt{1-\frac{2r+\frac{1}{6}}{(N+r+\frac{1}{2})}}}\int_{N}^{\infty}\frac{1}{(x+r+\frac{1}{2})^{r+\frac{1}{2}}}dx\nonumber\\
&=\frac{\mathbf{A}U_r}{(N-r+\frac{1}{2})\cdots(N-\frac{1}{2})(N+\frac{7}{2})\cdots(N+r-\frac{3}{2})(r-\frac{1}{2})\sqrt{N-r+\frac{1}{3}}}\nonumber\\
&=\frac{\mathbf{A}(N+\frac{3}{2})U_r}{\prod_{j=1}^{r}(N-r+2j-\frac{3}{2})(r-\frac{1}{2})\sqrt{N-r+\frac{1}{3}}}.\nonumber
\end{align} 
By a similar proof for even $ r $, we obtain
\begin{align*}
&\Vert f(x)-f_N(x)\Vert_{\infty}\leq\frac{\mathbf{A}(N+\frac{1}{2})U_r}{\prod_{j=1}^{r}(N-r+2j-\frac{3}{2})(r-\frac{1}{2})\sqrt{N-r+\frac{1}{3}}}.
\end{align*}
This complete the proof. 
\end{proof}
\section{New results on the Legendre coefficients}
In \citep{Xi20} the authors obtained an optimal upper bound for the Legendre coefficients without the assumptions in (\ref{VR1}) and (\ref{VR}) as follow 
\begin{equation}
\vert a_n\vert\leq \frac{U_r(n+\frac{1}{2})\Gamma(\frac{n-r}{2})}{2^r\sqrt{\pi}(n+r+1)\Gamma(\frac{n+r+1}{2})},
	\label{Liu}
\end{equation}
where, $ U_r=\int_{-1}^{1}|f^{(r+1)}(x)|dx $. This upper bound is really sharp and the authors provided this bound using the Rodrigues' formula. 

In Lemma \eqref{Key}, using the Sturm-Liouville problem, we presented a new relation between derivatives of Legendre polynomials. Now, using this lemma, we want to provide a new upper bound for the Legendre coefficients and we will show that this new upper bound match the ones obtained in  \citep{Xi20}. First, we state a lemma which provides an upper bound for Gegenbauer polynomials.
\begin{lemma}\label{Dur}\citep{Durand75}
Let $ \lambda\geq1. $ Then 
\begin{equation}
(1-x^2)^{\lambda-\frac{1}{2}}\left\vert C_{n}^{(\lambda)}(x)\right\vert\leq \frac{\Gamma(\frac{n}{2}+\lambda)}{\Gamma(\lambda)\Gamma(\frac{n}{2}+1)},\quad -1\leq x\leq1,
\end{equation}
where $ C_{n}^{(\lambda)}(x) $ is the Gegenbauer polynomial of degree $ n $. 
\end{lemma}

It is well-known that $ L_n(x)=C_{n}^{(\frac{1}{2})}(x) $. From \citep[p. 992]{Gradshteyn1980}, we know that 
\begin{equation}\label{d1}
\frac{d}{dx}C_{n}^{(\lambda)}(x)=2\lambda C_{n-1}^{(\lambda+1)}(x).
\end{equation}
So, we can see that 
\begin{equation}\label{d1}
\frac{d}{dx}L_{n}(x)=\frac{d}{dx}C_{n}^{(\frac{1}{2})}(x)= C_{n-1}^{(\frac{3}{2})}(x).
\end{equation}
From \eqref{sim}, Lemma \eqref{Dur} and \eqref{d1} we obtain
\begin{equation}\label{eqb2}
\vert \mathcal{L}_n(x)\vert\leq \frac{\Gamma(\frac{n+2}{2})}{n(n+1)\Gamma(\frac{3}{2})\Gamma(\frac{n+1}{2})}=\frac{\Gamma(\frac{n}{2})}{\sqrt{\pi}(n+1)\Gamma(\frac{n+1}{2})},\quad -1\leq x\leq1.
\end{equation}

This upper bound is sharp and interesting, because by using  \eqref{eqb2}, we obtain the same upper bound for Legendre coefficients as in  \citep{Xi20}.
Now, we present the following lemma. 
\begin{lemma}\label{cor}
Under the assumptions of Lemma \eqref{Key}, and using the inequlity \eqref{eqb2}, the following relation holds:
\begin{equation*}\label{sharp}
\left(\frac{\Gamma(\frac{n-r}{2})}{\sqrt{\pi}\lambda^1_{n,r}(n-r+1)\Gamma(\frac{n-r+1}{2})}+ \cdots 
+\frac{\Gamma(\frac{n+r}{2})}{\sqrt{\pi}\lambda^{2^r}_{n,r}(n+r+1)\Gamma(\frac{n+r+1}{2})}\right)=\frac{\Gamma(\frac{n-r}{2})}{2^r\sqrt{\pi}(n+r+1)\Gamma(\frac{n+r+1}{2})},
\end{equation*}
where, $ \lambda^{1}_{n,0}=1 $ and for $ r\geq1 $, $ \lambda^{j}_{n,r}, j=1,\ldots,2^r $ are defined in Lemma \eqref{Key}.
\end{lemma}
\begin{proof}
The equality is evident for $ r=0 $. In the case that $ r=1 $, we get
\begin{align*}
&\frac{\Gamma(\frac{n-1}{2})}{\sqrt{\pi}\lambda^{1}_{n,1}n\Gamma(\frac{n}{2})}+\frac{\Gamma(\frac{n+1}{2})}{\sqrt{\pi}\lambda^{2}_{n,1}(n+2)\Gamma(\frac{n+2}{2})}
=\frac{\Gamma(\frac{n-1}{2})}{2\sqrt{\pi}\lambda^{1}_{n,1}\Gamma(\frac{n+2}{2})}+\frac{(n-1)\Gamma(\frac{n-1}{2})}{2\sqrt{\pi}\lambda^{2}_{n,1}(n+2)\Gamma(\frac{n+2}{2})}\nonumber\\
&=\frac{\Gamma(\frac{n-1}{2})}{2\sqrt{\pi}\Gamma(\frac{n+2}{2})}\left(\frac{1}{\lambda^{1}_{n,1}}+\frac{(n-1)}{\lambda^{2}_{n,1}(n+2)}\right)=\frac{\Gamma(\frac{n-1}{2})}{2\sqrt{\pi}(n+2)\Gamma(\frac{n+2}{2})},
\end{align*}
where we used $ \Gamma(t+1)=t\Gamma(t) $. Also, in the case that $ r=2 $, the following equality is obtained.
\begin{align*}
&\frac{\Gamma(\frac{n}{2}-1)}{\sqrt{\pi}\lambda^{1}_{n,2}(n-1)\Gamma(\frac{n}{2}-\frac{1}{2})}+\frac{\Gamma(\frac{n}{2})}{\sqrt{\pi}\lambda^{2}_{n,2}(n+1)\Gamma(\frac{n}{2}+\frac{1}{2})}
+\frac{\Gamma(\frac{n}{2})}{\sqrt{\pi}\lambda^{3}_{n,2}(n+1)\Gamma(\frac{n}{2}+\frac{1}{2})}+\frac{\Gamma(\frac{n}{2}+\frac{1}{2})}{\sqrt{\pi}\lambda^{4}_{n,2}(n+3)\Gamma(\frac{n}{2}+\frac{3}{2})}\\
&=\underbrace{\frac{\Gamma(\frac{n}{2}-1)}{2\sqrt{\pi}\lambda^{1}_{n,2}\Gamma(\frac{n}{2}+\frac{1}{2})}+\frac{(n-2)\Gamma(\frac{n}{2}-1)}{2\sqrt{\pi}\lambda^{2}_{n,2}(n+1)\Gamma(\frac{n}{2}+\frac{1}{2})}}
+\underbrace{\frac{\Gamma(\frac{n}{2})}{2\sqrt{\pi}\lambda^{3}_{n,2}\Gamma(\frac{n}{2}+\frac{3}{2})}+\frac{(n-1)\Gamma(\frac{n}{2}-\frac{1}{2})}{2\sqrt{\pi}\lambda^{4}_{n,2}(n+3)\Gamma(\frac{n}{2}+\frac{3}{2})}}\\
&=\frac{\Gamma(\frac{n}{2}-1)(2n-1)}{2\sqrt{\pi}\lambda^{1}_{n,2}(n+1)\Gamma(\frac{n}{2}+\frac{1}{2})}+\frac{\Gamma(\frac{n}{2})(2n+3)}{2\sqrt{\pi}\lambda^{3}_{n,2}(n+3)\Gamma(\frac{n}{2}+\frac{3}{2})}\\
&=\frac{\Gamma(\frac{n}{2}-1)}{4\sqrt{\pi}(2n+1)\Gamma(\frac{n}{2}+\frac{3}{2})}+\frac{(n-2)\Gamma(\frac{n}{2}-1)}{4\sqrt{\pi}(2n+1)(n+3)\Gamma(\frac{n}{2}+\frac{3}{2})}\nonumber\\
&=\frac{\Gamma(\frac{n-2}{2})}{4\sqrt{\pi}(n+3)\Gamma(\frac{n+3}{2})}.
\end{align*}
Continuing the process for  $ r\geq3 $ show that the result holds.
\end{proof}

Applying \eqref{eqb2} in the Lemma \eqref{Key} and by similar proof of Theorem \eqref{Th}, we obtain the following theorem.

\begin{theo}\label{Thj}
Under the assumptions of Theorem \eqref{Th}, the upper bound for the Legendre coefficients of the function $ f $ for $ n\geq r $ is as follows:
\begin{equation*} 
\left\vert a_{n}\right\vert\leq\frac{U_r(n+\frac{1}{2})\Gamma(\frac{n-r}{2})}{2^r\sqrt{\pi}(n+r+1)\Gamma(\frac{n+r+1}{2})}.
\end{equation*}
\end{theo} 

The upper bound obtained in Theorem \eqref{Thj} match the ones obtained in \citep{Xi20}.  
Note that a new upper bounds for $ \mathcal{L}_n(x) $ help us to obtain a new upper bound for Legendre coefficients and improve the result in Theorem  \eqref{Thj}. 

\section{Comparison Results}\label{sec3}
This section aims to draw a comparison between our proposed upper bounds and the upper bounds provided so far.
In \citep{Hai21}, the author presented the following upper bound for the Legendre coefficients without the assumptions in (\ref{VR1}) and (\ref{VR})
\begin{equation}\label{eqw}
\vert a_n\vert\leq\frac{2U_r\prod_{j=1}^{r}h_{n-j}}{\sqrt{2\pi(n-r)}},
\end{equation}
where, $ h_{n-j}=\frac{1}{n-j+\frac{1}{2}} $. The author showed that the approximation errors obtained in \citep{Hai21} is sharper than ones presented in \citep{Shu12,Hai18}.

\begin{remark}
Now, we compare our proposed upper bound (\ref{HCo}) with the upper bound (\ref{eqw}). 

Assume that 
\begin{equation*}
\gamma_{n,r}=\frac{\mathbf{A}(n+\frac{1}{2})}{\Pi_{j=1}^{r+1}(n-r+2j-\frac{3}{2})\sqrt{n-r+\frac{1}{3}}},\quad\quad \theta_{n,r}=\frac{2\prod_{j=1}^{r}h_{n-j}}{\sqrt{2\pi(n-r)}}.
\end{equation*}
The numerical results for the values of $ \gamma_{n,r} $ and $ \theta_{n,r} $ are listed in Table \ref{t1}. The numerical results obtained from this table indicate that in all cases, the value of $\gamma_{n,r}$ is smaller than thats of $ \theta_{n,r} $. In particular, for the case that $ r$ is close to $n$, the value of $\gamma_{n,r}$ is much smaller than $ \theta_{n,r} $. Therefore, it reveals from this remark that the proposed upper bound \eqref{HCo} is sharper than the upper bound \eqref{eqw}.
\end{remark}

\begin{table}[!htbp] 
\centering
\caption{Comparison between $ \gamma_{n,r} $ and $ \theta_{n,r} $.}
\scalebox{0.9}{\begin{tabular}{ c c c c }\hline
n& r & $ \gamma_{n,r} $ &$ \theta_{n,r} $ \\ \hline 
5&4 &0.0023& 0.0135\\ 
10& 2 &0.0033& 0.0035\\ 
10&7 &$ 9.96\times10^{-8} $&$ 1.35\times10^{-6} $ \\
20&5&$ 7.68\times10^{-8} $&$ 1.27\times10^{-7} $\\ 
20&14&$6.81 \times10^{-19} $&$ 1.73\times10^{-16} $\\
30&15&$ 3.09\times10^{-23} $&$ 1.43\times10^{-21} $\\ 
30&24&$ 3.69\times10^{-35} $&$ 1.94\times10^{-30} $\\ 
\end{tabular}}
\label{t1}
\end{table}

In the following, we compare the approximation error of the partial sums of Legendre polynomials evaluated by Theorem \ref{Th2} and the bound presented in \citep{Hai21}. The author proved that
\begin{equation}
\Vert f(x)-f_N(x)\Vert_{\infty}\leq \left\{
		\begin{array}{lr}
	\frac{4U_1}{\sqrt{2\pi(N-1)}}&r=1,\\
			\frac{2U_r}{\prod_{j=2}^{r}(N-j+\frac{3}{2})(r-1)\sqrt{2\pi(N-r+1)}}& r\geq2,
		\end{array}\right.
	\label{wang2}
\end{equation}
which is sharper than ones presented in \citep{Shu12,Hai18}.
\begin{example}\label{ex1}
Let $ j\geq2 $ and $ -1<t<1. $ Consider the function $ f_j(x)=\frac{1}{j!}(x-t)^{j-1}\vert x-t\vert$. This function and its derivatives are absolutely continuous on $ [-1,1] $ and $ f^{(j-1)}_j(x)=\vert x-t\vert $. Also,
$f_j^{(j)}(x)=2H(x-t)-1, $
where $ H(x) $ is the Heaviside step function which is of bounded variation and  
$
f_j^{(j+1)}(x)=2\delta(x-t),
$
where $ \delta(x-t) $ is the Dirac delta function. Then
\begin{equation*}
U_{j}=\int_{-1}^{1}\left\vert f_j^{(j+1)}(x)\right\vert dx=\int_{-1}^{1}2\delta(x-t)dx=2,\\
\end{equation*}
In Table \ref{T2} the comparison results between Theorem \ref{Th2} and \eqref{wang2} are listed  for some values of $ N, j $.  
\end{example}
\begin{table}[!htbp]
\centering
\caption{Comparison results between two approximation errors for the function $ f_j(x) $.}
\scalebox{0.65}{\begin{tabular}{c c c c ||c c c c c}\hline
$N$ & $j$& Inequality (\ref{wang2}) & Inequality (\ref{Co}) & $N$ & $j$ & Inequality (\ref{wang2}) & Inequality (\ref{Co}) \\ \hline
15&3&$ 1.394\times10^{-3} $&$ 9.868\times10^{-4} $&20& 3&$ 6.069\times10^{-4} $&$ 4.421\times10^{-4} $\\
15&10&$ 1.895\times10^{-10} $&$ 6.824\times10^{-12} $&20& 10&$ 2.352\times10^{-12} $&$ 2.126\times10^{-13} $\\ 
15&12&$ 1.321\times10^{-11} $&$ 7.989\times10^{-14} $&20&12 &$ 2.682\times10^{-14} $&$ 7.544\times10^{-16} $\\ 
\hline
\end{tabular}}
\label{T2}
\end{table} 

\section{Conclusion}
In this paper  a new relation between derivatives of Legendre polynomials is presented in Lemma  \ref{Key} and by using this key lemma and Lemmas \ref{le}, \ref{l3},  new upper bounds for the Legendre coefficients of differentiable functions are obtained (see Theorem \ref{Th}).  We compare these upper bounds by previous upper  bounds which are presented in \citep{ Hai21}. Moreover, we provide a new upper bound error on the approximation of a function $ f (x) $ by truncated Legendre polynomial series.

\end{document}